\documentclass[pdflatex,sn-mathphys-num]{sn-jnl}

\usepackage[T1]{fontenc}
\usepackage{lmodern}
\usepackage{amsmath,amssymb,amsfonts}
\usepackage{amsthm}
\usepackage{mathtools}
\usepackage{bm}
\usepackage{enumitem}
\usepackage{microtype}

\newcommand{\bb}{\mathbb}
\newcommand{\R}{\bb R}
\newcommand{\N}{\bb N}
\newcommand{\cF}{\mathcal F}
\newcommand{\cH}{\mathcal H}
\newcommand{\eps}{\varepsilon}
\newcommand{\zero}{\bm 0}
\newcommand{\e}{\bm e}
\newcommand{\x}{\bm x}
\newcommand{\y}{\bm y}
\newcommand{\z}{\bm z}
\newcommand{\w}{\bm w}
\newcommand{\vve}{\bm v}
\newcommand{\p}{\bm p}
\newcommand{\q}{\bm q}
\newcommand{\thetaa}{\bm\theta}
\newcommand{\one}{\bm 1}
\newcommand{\val}{\textup{val}}
\newcommand{\intrvol}[1]{\operatorname{vol}_{#1}}
\newcommand{\ip}[2]{\left\langle #1,#2\right\rangle}
\newcommand{\norm}[1]{\left\lVert #1\right\rVert_2}
\newcommand{\abs}[1]{\left\lvert #1\right\rvert}
\newcommand{\dist}{\operatorname{dist}}
\newcommand{\conv}{\operatorname{conv}}
\newcommand{\aff}{\operatorname{aff}}
\newcommand{\lin}{\operatorname{lin}}
\newcommand{\proj}{\operatorname{proj}}
\newcommand{\argminop}{\operatorname*{argmin}}
\newcommand{\argmaxop}{\operatorname*{argmax}}
\newcommand{\widetildeTheta}{\widetilde{\Theta}}

\newtheoremstyle{thmstyleone}%
{12pt}{12pt}{\small\itshape}{0pt}{\small\bfseries}{}{.5em}{}
\newtheoremstyle{thmstyletwo}%
{12pt}{12pt}{\small\normalfont}{0pt}{\small\itshape}{}{.5em}{}
\newtheoremstyle{thmstylethree}%
{12pt}{12pt}{\small\normalfont}{0pt}{\small\bfseries}{}{.5em}{}

\theoremstyle{thmstyleone}
\newtheorem{theorem}{Theorem}
\newtheorem{proposition}[theorem]{Proposition}
\newtheorem{lemma}[theorem]{Lemma}
\newtheorem{corollary}[theorem]{Corollary}

\theoremstyle{thmstyletwo}

\theoremstyle{thmstylethree}
\newtheorem{definition}{Definition}

\begin{document}

\hypersetup{
  pdftitle={Closing the Oracle-Complexity Gap in Derivative-Free Convex Optimization: A Near-Quadratic Lower Bound from Exact Function Values},
  pdfauthor={Anonymous Authors},
  pdfkeywords={convex optimization, information complexity, value oracle, nonsmooth optimization, convex geometry},
  bookmarksdepth=3
}

\title[Near-quadratic exact-value complexity]{Closing the Oracle-Complexity Gap in Derivative-Free Convex Optimization: A Near-Quadratic Lower Bound from Exact Function Values}

\author[1]{\fnm{Phillip} \sur{Kerger}}
\affil[1]{\orgdiv{Department of Industrial Engineering and Operations Research, UC Berkeley}}

\abstract{
We study the deterministic query complexity of minimizing a convex Lipschitz function over a $d$-dimensional Euclidean ball using only exact function values.  At accuracy $\Theta(d^{-1/2})$, the previously applicable lower bound was $\Omega(d)$, inherited from the stronger full first-order oracle, while an upper bound from Protasov's value-only method requires $O(d^2\log^2 d)$ evaluations. By providing a lower bound of $\Omega(\,\frac{d^2}{\log(d+1)})$ on the oracle complexity in this setting, we thereby close this gap dating back to 1996, up to polylogarithmic factors. Furthermore, we are able to lift this result to the mixed-integer setting: Mixed-integer convex optimization with $d$ continuous and $n$ discrete variables using function values requires $\tilde{\Omega}(d^2\cdot 2^n)$ queries. 
}


\maketitle

\section{Introduction}\label{sec:intro}

Optimization must often proceed without access to derivatives.  In black-box and simulation-based optimization, evaluating a candidate design may require running a physical experiment or an expensive simulator~\cite{ConnScheinbergVicente2009}.  In bandit optimization, the learner observes the loss only at the action or actions it selects~\cite{DuchiJordanWainwrightWibisono2015,Shamir2013,Shamir2017}.  When such evaluations dominate the cost of optimization, a fundamental question is how many of them are needed to find an approximately optimal point.  We address a clean worst-case version of this question for deterministic convex optimization.

Oracle complexity formalizes this question by isolating the number of observations of an unknown objective.  The foundational work of Nemirovski and Yudin~\cite{NemirovskiYudin1983} established the theory for first-order models, in which a query returns a separating hyperplane, a subgradient, or a value--subgradient pair.  For the weaker deterministic function-value-only model, however, the best bounds have remained separated by a polynomial factor.  We study its most informative noiseless form: each query returns the exact objective value.  We impose no finite-precision, time, or memory restriction on the algorithm and study purely the number of function evaluations required.

Specifically, let
 $B_d:=\{\x\in\R^d:\norm{\x}\leq1\}$
and let $\cF_d$ contain the convex, $1$-Lipschitz functions $f:B_d\to\R$ with $f(\zero)=0$.  We write $Q_{\val}^{\det}(d,\eps)$ for the least number of exact-value queries sufficient, in the worst case over $f\in\cF_d$, to return $\widehat\x\in B_d$ satisfying
\[
 f(\widehat\x)-\min_{\x\in B_d}f(\x)\leq\eps.
\]
Section~\ref{sec:model} gives the formal model.

\subsection{Existing bounds and neighboring models}\label{subsec:published-gap}

Let $Q_{\mathrm{FO}}^{\det}(d,\eps)$ denote the corresponding complexity for globally defined convex $1$-Lipschitz objectives when each query returns the value and an arbitrary oracle-selected subgradient.  In the standard sequential model, projected subgradient descent and the classical Nemirovski--Yudin lower bound \cite{NemirovskiYudin1983} imply, for every fixed $\beta>0$,
\begin{equation}\label{eq:fo-benchmark}
 Q_{\mathrm{FO}}^{\det}\!\left(d,\frac{\beta}{\sqrt d}\right)=\Theta(d);
\end{equation}
see~\cite{NemirovskiYudin1983,Bubeck2015,BubeckJiangLeeLiSidford2019}.  More generally, the classical lower bound is $\Omega(\min\{d,\eps^{-2}\})$.  Because the restriction of every such global objective belongs to $\cF_d$ and a full first-order response contains the value, any value-only algorithm could ignore the subgradient.  Thus, before this work, the best lower bound applicable to unrestricted deterministic exact-value algorithms at accuracy $\Theta(d^{-1/2})$ was only $\Omega(d)$.

Memory-query tradeoffs for first-order convex optimization, initiated as an explicit question by Woodworth and Srebro~\cite{WoodworthSrebro2019}, give superlinear query lower bounds when the algorithm has subquadratic memory~\cite{MarsdenSharanSidfordValiant2024,BlanchardZhangJaillet2024}.  These results do not directly strengthen the preceding bound because our algorithms may use unlimited memory.  Nor can one place a $T$-query algorithm within the scope of these theorems merely by storing its transcript: an exact value is an arbitrary real number with potentially unbounded bit complexity, and later queries may depend discontinuously on earlier values.  Thus, without an additional model-preserving compression argument, the value transcript need not have a representation using $\widetilde O(T)$ bits.

In the opposite direction, Protasov~\cite{Protasov1996} gives a deterministic exact-value method needing
\[
 O\!\left(d^2\log(d+1)\log\frac1\eps\right)
\]
evaluations for relative objective error $\eps$, yielding $O(d^2\log^2(d+1))$ evaluations at accuracy $\eps = \Theta(d^{-1/2})$.  Thus the deterministic exact-value bounds have had a polynomial gap, $\Omega(d)$ versus $\tilde{O}(d^2)$, dating back to this 1996 upper bound by Protasov.

Several neighboring zeroth-order models have different complexities with tight existing results.  With randomized two-point feedback, algorithms achieve expected error $O(LR\sqrt{d/T})$ after $T$ two-point rounds, and matching minimax lower bounds hold in the standard stochastic or online formulations~\cite{DuchiJordanWainwrightWibisono2015,Shamir2017}.  Under noisy one-point feedback, even normalized smooth strongly convex problems can have minimax expected error of order $\min\{1,d/\sqrt T\}$~\cite{Shamir2013}.  For smooth, generally nonconvex derivative-free optimization, probabilistically accurate models and noisy trust-region methods admit iteration-complexity and high-probability guarantees under assumptions tailored to those settings~\cite{ConnScheinbergVicente2009,CartisScheinberg2018,CaoBerahasScheinberg2024}.  These results are not directly comparable to the noiseless, nonsmooth, deterministic, exact-real model and therefore do not transfer to our setting.

Basu, Kerger, and Molinaro~\cite{BasuKergerMolinaro2025} recently proved near-quadratic lower bounds for bit and inner-product first-order oracles and explicitly noted that a comparable bound was unknown even when the oracle reveals full function values.  The responses of the oracles studied there provide first-order information but only finitely many bits; an exact value supplies no subgradient or separating normal but may encode unboundedly many bits.  Neither result therefore transfers automatically to the other model.

\subsection{Our contribution}\label{subsec:contribution}

We close the deterministic exact-value dimension gap up to logarithmic factors by providing a significantly stronger lower bound.

\begin{theorem}\label{thm:main-intro}
There are universal constants $c,C,\eps_0>0$ and $d_0\in\N$ such that, for every $d\geq d_0$,
\[
 c\,\frac{d^2}{\log(d+1)}
 \leq
 Q_{\val}^{\det}\!\left(d,\frac{\eps_0}{\sqrt d}\right)
 \leq
 C d^2\log^2(d+1).
\]
In particular,
\[
 Q_{\val}^{\det}\!\left(d,\frac{\eps_0}{\sqrt d}\right)
 =\widetildeTheta(d^2).
\]
\end{theorem}

  Together with~\eqref{eq:fo-benchmark}, it gives a polynomial separation between deterministic exact-value and full first-order information at the accuracy constant supplied by the theorem.  
  To simplify notation, we take $d$ to be even throughout this paper. Scaling extends the result to an $L$-Lipschitz objective on a ball of radius $R$ at accuracy $\eps_0LR/\sqrt d$; monotonicity gives the same lower bound at smaller errors.

We also obtain a near-optimal mixed-integer consequence of the transfer theorem of Basu, Jiang, Kerger, and Molinaro~\cite[Theorem~7]{BasuJiangKergerMolinaro2025}.  On $[0,1]^n\times[-1,1]^d$, suppose that the first $n$ variables are restricted to be integral, each binary fiber is Euclidean $1$-Lipschitz in the $d$ continuous variables, and value queries may be made at arbitrary fractional points of the product domain.  Write $Q_{\mathrm{MI},\val}^{\det}(n,d,\eps)$ for the deterministic exact-value complexity of this problem; Section~\ref{sec:mixed-transfer} gives the full formal model.

\begin{corollary}[Mixed-integer consequence]\label{cor:mixed-value}
There are universal constants $c,C,\eps_{\mathrm{MI}}>0$ and $d_0\in\N$ such that, for every $n\geq1$ and $d\geq d_0$,
\[
 c\,2^n\frac{d^2}{\log(d+1)}
 \leq
 Q_{\mathrm{MI},\val}^{\det}\!\left(n,d,\frac{\eps_{\mathrm{MI}}}{\sqrt d}\right)
 \leq
 C\,2^n d^2\log^2(d+1).
\]
Consequently,
\[
 Q_{\mathrm{MI},\val}^{\det}\!\left(n,d,\frac{\eps_{\mathrm{MI}}}{\sqrt d}\right)
 =\widetildeTheta(2^n d^2).
\]
\end{corollary}

The lower bound is not a direct application of the transfer theorem: its continuous source instances must have a common optimal value and must use a box as their full feasible region.  Section~\ref{sec:mixed-transfer} supplies a constant-size binning and flooring argument for the first issue and a radial extension for the second.  The upper bound follows by enumerating the $2^n$ fibers and applying Protasov's method on each one.

Here $\widetildeTheta$ suppresses fixed powers of $\log(d+1)$.  The continuous theorem establishes one universal constant multiple of $d^{-1/2}$, not every such multiple.  The continuous upper bound also implies near-quadratic dimension dependence at every fixed polynomially smaller accuracy; see Corollary~\ref{cor:polynomial-accuracy}.

\subsection{Proof overview}\label{subsec:overview}

We first provide an overview of the proof, which will be presented in full detail in section \ref{sec:lower-proof}.

For the entirety of this paper, assume $d$ is even, set $m=d/2$, and write points as $(\x,\z)\in\R^m\times\R^m$.  The proof has four main components to it:

\begin{enumerate}[leftmargin=22pt,label=(\roman*)]
\item \emph{Hard family of functions.}  For rows $\w_i$ in a ball of radius $\tau=\Theta(m^{-1/2})$, consider
\[
 f_W(\x,\z)=\max_{i\in[m]}\{a x_i+\ip{\w_i}{\z}\},
 \qquad a=\tfrac12.
\]
These functions are convex and $1$-Lipschitz.  As support functions of explicit polytopes, their minimizers are determined by the points of those polytopes closest to the origin.

\item \emph{Exact resisting oracle.}  For each row, the adversary maintains a compact convex uncertainty set $P_i\subseteq\tau B_m$, with the invariant that every matrix in $P_1\times\cdots\times P_m$ reproduces the entire transcript.  A query either cuts one row by a large affine section, reducing its dimension by one, or only truncates rows by halfspaces that remove at most a $1/(4m)$ fraction of intrinsic volume.

\item \emph{Aggregate width.}  After $T=O(m^2/\log m)$ queries, the total codimension and logarithmic normalized-volume loss are small.  Hence linearly many row sets remain high-dimensional with large volume radius.  Urysohn's inequality and averaging yield one ambient direction in which their aggregate width is $\Omega(m\tau)$.

\item \emph{Indistinguishable minimizers.}  Moving all wide rows to opposite extremes in that common direction produces two transcript-compatible functions.  Their closest support-polytope points use nearly uniform barycentric weights, so the aggregate row displacement separates the two minimizers by a constant.  The support-function growth inequality then forces objective error $\Omega(m^{-1/2})=\Omega(d^{-1/2})$ for the algorithm's common output.
\end{enumerate}

\subsection{AI-assisted development and formal verification in Lean}\label{subsec:ai-usage}

Modern AI tools played a substantial role in developing the mathematical arguments in this work, and it is accurate to say that the AI model used solved the problem, not the author of this paper. In particular, GPT~5.6 Sol Pro was used following a workflow similar to that documented by OpenAI in connection with its recent preprint on the Cycle Double Cover Conjecture and accompanying prompt~\cite{OpenAI2026CDC,OpenAI2026CDCPrompt}. Similar success in the optimization literature has recently been achieved in \cite{ErnestRyu_AI-assisted-nesterov-proof}, on point convergence of Nesterov's accelerated gradient method.
For our work, this process first succeeded in proving a $\tilde{\Omega}(d^2)$ lower bound at accuracy $\epsilon$ of order $d^{-3}$, which was subsequently refined to the order-$d^{-1/2}$ result that we present here. Appendix~\ref{appendix} provides an account of the workflow as well as prompts used and answers received (including links to chats), which were followed by human verification by the author. The author takes full responsibility for every statement and proof in the manuscript. 

Furthermore, to cast aside any (well-warranted) skepticism of AI-assisted mathematical proofs, we provide formal verification of the proof of the $\tilde{\Omega}(d^2)$ lower bound using the Lean programming language, which can be found on the author's \href{https://github.com/PhillipKerger/zero-order-bounds-lean-verification}{Github page}: \begin{verbatim}
    https://github.com/PhillipKerger/zero-order-bounds-lean-verification
\end{verbatim}

\section{Oracle model and main results}\label{sec:model}

It is convenient to state the result with arbitrary radius and Lipschitz constant.  Throughout, $[m]:=\{1,\ldots,m\}$ for $m\in\N$, all logarithms are natural, and all Euclidean spaces carry their standard inner products.  For $R,L>0$, let
\[
 B_d(R):=\{\x\in\R^d:\norm{\x}\leq R\}
\]
and let $\cF_d(R,L)$ be the class of functions $f:B_d(R)\to\R$ such that
\begin{enumerate}[label=(\roman*),leftmargin=22pt]
\item $f$ is convex on $B_d(R)$;
\item $\abs{f(\x)-f(\y)}\leq L\norm{\x-\y}$ for all $\x,\y\in B_d(R)$;
\item $f(\zero)=0$.
\end{enumerate}
No differentiability, smoothness, strict or strong convexity, polyhedrality, or finite representation is assumed.

\begin{definition}[Deterministic exact-value complexity]\label{def:complexity}
A deterministic exact-value algorithm adaptively chooses points $\x_t\in B_d(R)$.  The point $\x_t$ may be an arbitrary deterministic function of $d,R,L,\eps$ and the exact previous answers $f(\x_1),\ldots,f(\x_{t-1})$.  After at most $T$ queries, the algorithm outputs an arbitrary point $\widehat\x\in B_d(R)$, which need not have been queried.  The algorithm may use unlimited computation and exact real arithmetic, and its query and output maps may be discontinuous.  These maps are fixed independently of the unknown objective, which is accessed only through the value oracle.

We denote by $Q_{\val}^{\det}(d,R,L,\eps)$ the smallest integer $T$ for which such an algorithm guarantees
\[
 f(\widehat\x)-\min_{\x\in B_d(R)}f(\x)\leq\eps
 \qquad\text{for every }f\in\cF_d(R,L).
\]
We abbreviate $Q_{\val}^{\det}(d,\eps):=Q_{\val}^{\det}(d,1,1,\eps)$.
\end{definition}

For a fixed algorithm, a \emph{transcript} is the sequence of query--response pairs it generates.  A function is \emph{consistent} with a transcript if it attains every recorded value at the corresponding query.  Thus a lower bound follows from a transcript admitting two consistent functions with disjoint sets of $\eps$-accurate solutions.

Every function evaluation counts, including repeated queries, line-search trials, rejected points, and evaluations in a deterministic batch; a batch can be serialized without changing its size.  Queries must lie in the ball, while computation involving the known domain and internally maintained sets is free.

\begin{lemma}[Scaling]\label{lem:scaling}
For every $d\geq2$, $R,L>0$, and $\eps>0$,
\[
 Q_{\val}^{\det}(d,R,L,\eps)
 =Q_{\val}^{\det}\!\left(d,\frac{\eps}{LR}\right).
\]
\end{lemma}

\begin{proof}
For $f\in\cF_d(R,L)$, define $g:B_d\to\R$ by
\[
 g(\y):=\frac{f(R\y)}{LR}.
\]
Then $g\in\cF_d$, and a query to $g$ at $\y$ is simulated by one query to $f$ at $R\y$.  Objective gaps are divided by $LR$.  The inverse transformation $f(\x)=LR\,g(\x/R)$ proves the reverse inequality.
\end{proof}

\begin{theorem}[Deterministic lower bound]\label{thm:lower}
There are universal constants $c,\eps_0>0$ and $d_0\in\N$ such that, for every $d\geq d_0$,
\[
 Q_{\val}^{\det}\!\left(d,\frac{\eps_0}{\sqrt d}\right)
 \geq c\frac{d^2}{\log(d+1)}.
\]
When $d$ is even, the hard instance may be chosen as a maximum of $d/2$ linear functions.  One may take $\eps_0=10^{-7}$.
\end{theorem}

\paragraph{Known upper bound.}
Protasov~\cite{Protasov1996} proves that, for a continuous convex function on a known $d$-dimensional convex body, deterministic access to exact function values suffices to obtain relative objective error $\delta\in(0,1/2)$ using
\[
 O\!\left(d^2\log(d+1)\log\frac1\delta\right)
\]
evaluations.  Since every $f\in\cF_d(R,L)$ has objective range at most $2LR$, his result gives a universal constant $C_1>0$ such that, for $0<\eps\leq LR$,
\begin{equation}\label{eq:protasov-upper}
 Q_{\val}^{\det}(d,R,L,\eps)
 \leq
 C_1d^2\log(d+1)\log\!\left(\frac{4LR}{\eps}\right).
\end{equation}
Theorem~\ref{thm:main-intro} follows from Theorem~\ref{thm:lower} and~\eqref{eq:protasov-upper} with $R=L=1$.

\begin{corollary}[Polynomially small accuracy]\label{cor:polynomial-accuracy}
For every fixed $\alpha>1/2$, uniformly over $R,L>0$,
\[
 Q_{\val}^{\det}(d,R,L,LRd^{-\alpha})
 =\widetildeTheta(d^2),
\]
where the hidden constants and logarithmic powers may depend on $\alpha$ but not on $d$, $R$, or $L$.  At $\alpha=1/2$, the same conclusion holds for accuracy $\beta LR/\sqrt d$ whenever $0<\beta\leq\eps_0$ is fixed.
\end{corollary}

\begin{proof}
By Lemma~\ref{lem:scaling}, it suffices to work with $R=L=1$.  For $\alpha>1/2$, one has $d^{-\alpha}\leq\eps_0d^{-1/2}$ for all sufficiently large $d$, so monotonicity and Theorem~\ref{thm:lower} give the lower bound.  Equation~\eqref{eq:protasov-upper} gives the upper bound $O_\alpha(d^2\log^2(d+1))$.  The case $\alpha=1/2$ is identical when $\beta\leq\eps_0$.
\end{proof}

\section{Proof of Theorem~\ref*{thm:lower}}\label{sec:lower-proof}
This section is devoted to the proof of the claimed lower-bound of Theorem \ref{thm:lower}. We organize this proof into four subsections, following the presentation of \ref{subsec:overview}. 

It suffices to prove the theorem in even dimension.  To see this, let $d$ be odd and set $d':=d-1$.  For $f\in\cF_{d'}$, define $F:B_d\to\R$ by $F(x,t):=f(x)$.  A deterministic algorithm in dimension $d$ can then be simulated in dimension $d'$ by projecting every query $(x,t)$ to $x$ and projecting its final output in the same way.  Thus a lower bound in dimension $d'$ transfers to dimension $d$.  Since $(d')^2/\log(d'+1)=\Omega(d^2/\log(d+1))$ and $\eps_0/\sqrt d\leq\eps_0/\sqrt{d'}$, only the universal lower-bound constant changes.  Hence throughout this paper, $d$ is even and
\[
 m:=\frac d2.
\]

\subsection{The hard function class and its solutions}\label{subsec:hard-family}

Assume $d\geq4$ and write points $\q\in\R^d$ as
\[
 \q=(\x,\z)\in\R^m\times\R^m.
\]
Let
\[
 a:=\frac12,
 \qquad
 \Gamma:=100,
 \qquad
 \tau:=\frac{a}{\Gamma\sqrt m}.
\]
For a matrix $W$ with rows $\w_1,\ldots,\w_m\in\tau B_m$, define
\begin{equation}\label{eq:hard-function}
 f_W(\x,\z)
 :=\max_{i\in[m]}\{a x_i+\ip{\w_i}{\z}\}.
\end{equation}
We denote this max-affine family by
\begin{equation}\label{eq:hard-family}
 \cH_d:=\{f_W:W\in(\tau B_m)^m\}.
\end{equation}
Every member of $\cH_d$ is the maximum of $m=d/2$ linear functions.

\begin{lemma}[Membership of the hard family]\label{lem:hard-in-class}
For every $W\in(\tau B_m)^m$, the function $f_W$ belongs to $\cF_d$.
\end{lemma}

\begin{proof}
The function is a maximum of linear functions and is therefore convex.  Its slopes are
\[
 \vve_i=(a\e_i,\w_i),
 \qquad i\in[m],
\]
where $\e_i$ is the $i$th standard basis vector in $\R^m$.  Since
\[
 \norm{\vve_i}^2
 =a^2+\norm{\w_i}^2
 \leq a^2+\tau^2<1,
\]
$f_W$ is $1$-Lipschitz on $\R^d$, and hence on $B_d$.  Finally, $f_W(\zero)=0$.
\end{proof}\vspace{2em}

The optimizer has a useful support-function description.  For $W\in(\tau B_m)^m$, set
\[
 K_W:=\conv\{\vve_1,\ldots,\vve_m\}
\]
and let $\p_W$ be the Euclidean projection of the origin onto $K_W$.

\begin{lemma}[Support-function geometry]\label{lem:support-geometry}
For every $W\in(\tau B_m)^m$, one has $\p_W\neq\zero$.  The unique minimizer of $f_W$ over $B_d$ is
\begin{equation}\label{eq:qstar}
 \q_W^*:=-\frac{\p_W}{\norm{\p_W}},
\end{equation}
and
\begin{equation}\label{eq:min-value}
 \min_{\q\in B_d}f_W(\q)=-\norm{\p_W}.
\end{equation}
Moreover, for every $\q\in B_d$,
\begin{equation}\label{eq:quadratic-growth}
 f_W(\q)-f_W(\q_W^*)
 \geq
 \frac{\norm{\p_W}}2\norm{\q-\q_W^*}^2.
\end{equation}
\end{lemma}

\begin{proof}
Equation~\eqref{eq:hard-function} is the support function of $K_W$:
\[
 f_W(\q)=\max_{\vve\in K_W}\ip{\vve}{\q}.
\]
Every $\vve\in K_W$ has the form
\[
 \vve=\left(a\lambda,\sum_{i=1}^m\lambda_i\w_i\right)
 \quad\text{for some}\quad
 \lambda\in\Delta_m:=\{\lambda\in\R_+^m:\one^\top\lambda=1\}.
\]
Since $\norm{\lambda}\geq1/\sqrt m$, every point of $K_W$ has norm at least $a/\sqrt m$; in particular, $\p_W\neq\zero$.

The variational characterization of Euclidean projection gives
\[
 \ip{\p_W}{\vve-\p_W}\geq0
 \qquad\text{for every }\vve\in K_W.
\]
Thus, for $\q_W^*=-\p_W/\norm{\p_W}$,
\[
 \ip{\vve}{\q_W^*}\leq-\norm{\p_W}
 \qquad(\vve\in K_W),
\]
with equality at $\vve=\p_W$.  This proves~\eqref{eq:min-value}.  Conversely, for every $\q\in B_d$,
\[
 f_W(\q)\geq\ip{\p_W}{\q}\geq-\norm{\p_W}.
\]
Equality in the second inequality, with $\norm{\q}\leq1$, forces $\q=-\p_W/\norm{\p_W}$, proving uniqueness.

Finally,
\begin{align*}
 f_W(\q)-f_W(\q_W^*)
 &\geq \ip{\p_W}{\q}+\norm{\p_W}\\
 &=\norm{\p_W}\bigl(1-\ip{\q}{\q_W^*}\bigr).
\end{align*}
Because $\norm{\q}\leq1$ and $\norm{\q_W^*}=1$,
\[
 \norm{\q-\q_W^*}^2
 =\norm{\q}^2+1-2\ip{\q}{\q_W^*}
 \leq2\bigl(1-\ip{\q}{\q_W^*}\bigr),
\]
which proves~\eqref{eq:quadratic-growth}.
\end{proof}\vspace{2em}

We next show that the closest point is controlled by an almost-uniform average of the rows.  This stability is what later converts aggregate row width into minimizer separation.

\begin{lemma}[Barycentric structure of the projection]\label{lem:barycentric}
For $W\in(\tau B_m)^m$, there is a unique $\lambda(W)\in\Delta_m$ such that
\[
 \p_W=(a\lambda(W),\z_W),
 \qquad
 \z_W:=\sum_{i=1}^m\lambda_i(W)\w_i.
\]
Every coordinate of $\lambda(W)$ is positive.  If
\[
 \mu_W:=\frac1m\sum_{i=1}^m\w_i,
 \qquad
 \Sigma_W:=\frac1m\sum_{i=1}^m
 (\w_i-\mu_W)(\w_i-\mu_W)^\top,
\]
then
\begin{align}
 \lambda_i(W)
 &=\frac1m+
 \frac{\ip{\mu_W-\w_i}{\z_W}}{a^2},
 \label{eq:lambda-formula}\\
 \z_W
 &=\left(I+\frac m{a^2}\Sigma_W\right)^{-1}\mu_W.
 \label{eq:z-formula}
\end{align}
In particular,
\begin{equation}\label{eq:lambda-close}
 \abs{\lambda_i(W)-\frac1m}
 \leq\frac{2}{\Gamma^2m}
 \qquad(i\in[m]),
\end{equation}
and
\begin{equation}\label{eq:p-norm-range}
 \frac a{\sqrt m}
 \leq\norm{\p_W}
 \leq
 \frac a{\sqrt m}\sqrt{1+\Gamma^{-2}}.
\end{equation}
\end{lemma}

\begin{proof}
The barycentric vector $\lambda(W)$ is the unique minimizer over $\Delta_m$ of the strictly convex function
\begin{equation}\label{eq:lambda-objective}
 \Phi(\lambda)
 :=a^2\norm{\lambda}^2+
 \norm{\sum_{i=1}^m\lambda_i\w_i}^2.
\end{equation}
Uniqueness also follows directly from the first block $a\lambda$.

Let $\lambda$ be the minimizer and put $\z_W=\sum_i\lambda_i\w_i$.  Define
\[
 g_i:=a^2\lambda_i+\ip{\w_i}{\z_W},
\]
so that the $i$th coordinate of $\nabla\Phi(\lambda)$ is $2g_i$.  Suppose that $\lambda_i=0$ for some $i$.  Choose $j$ with $\lambda_j\geq1/m$.  Since $\norm{\z_W}\leq\tau$,
\[
 g_i\leq\tau^2,
 \qquad
 g_j\geq\frac{a^2}{m}-\tau^2>\tau^2,
\]
where the strict inequality uses $\tau^2=a^2/(\Gamma^2m)$ and $\Gamma^2>2$.  For all sufficiently small $h>0$, the vector $\lambda+h(\e_i-\e_j)$ remains in $\Delta_m$, and the directional derivative of $\Phi$ at $h=0$ equals $2(g_i-g_j)<0$.  This contradicts optimality.  Hence every $\lambda_i$ is positive.

The interior Lagrange multiplier equations give a scalar $\beta$ such that
\[
 a^2\lambda_i+\ip{\w_i}{\z_W}=\beta
 \qquad(i\in[m]).
\]
Averaging over $i$ gives $\beta=a^2/m+\ip{\mu_W}{\z_W}$, and subtraction yields~\eqref{eq:lambda-formula}.  Substituting this formula into $\z_W=\sum_i\lambda_i\w_i$ gives
\begin{align*}
 \z_W
 &=\mu_W+\frac1{a^2}
   \sum_{i=1}^m\w_i(\mu_W-\w_i)^\top\z_W\\
 &=\mu_W-\frac m{a^2}\Sigma_W\z_W,
\end{align*}
because
\[
 \sum_{i=1}^m\w_i(\mu_W-\w_i)^\top=-m\Sigma_W.
\]
Rearranging proves~\eqref{eq:z-formula}.

The matrix $I+(m/a^2)\Sigma_W$ is positive definite and its inverse has operator norm at most one.  Therefore
\[
 \norm{\z_W}\leq\norm{\mu_W}\leq\tau.
\]
Together with $\norm{\mu_W-\w_i}\leq2\tau$, equation~\eqref{eq:lambda-formula} gives
\[
 \abs{\lambda_i-\frac1m}
 \leq\frac{2\tau^2}{a^2}
 =\frac{2}{\Gamma^2m},
\]
which is~\eqref{eq:lambda-close}.

For the lower norm bound, $\norm{\p_W}\geq a\norm{\lambda(W)}\geq a/\sqrt m$.  For the upper bound, evaluate~\eqref{eq:lambda-objective} at $\one/m$:
\[
 \norm{\p_W}^2
 \leq\frac{a^2}{m}+\norm{\mu_W}^2
 \leq\frac{a^2}{m}(1+\Gamma^{-2}).
\]
This proves~\eqref{eq:p-norm-range}.
\end{proof}

\subsection{Construction of adversarial oracle responses}\label{subsec:resisting}

Fix an arbitrary deterministic algorithm.  We construct each oracle response from the transcript preceding the current query and later show that, after $\widetilde{\Omega}(d^2)$ queries, two functions in the hard family remain indistinguishable.  After query $t$, the adversary maintains one row uncertainty set for each affine piece,
\[
 P_i^t\subseteq\tau B_m,
 \qquad i\in[m].
\]
Intuitively, $P_i^t$ contains the row vectors that remain possible for the $i$th affine piece after the first $t$ responses.  The crucial invariant is rectangular: every independent selection of one row from each $P_i^t$ defines a fixed function in the hard family that agrees with the entire transcript so far.  Initially no information about any row has been revealed, so $P_i^0=\tau B_m$ for every $i$.

If $P\subseteq\R^m$ is a nonempty compact convex set, we write $\dim P$ to denote its affine dimension.  On every $k$-dimensional affine subspace we use the $k$-dimensional Lebesgue measure induced by the Euclidean metric, and denote it by $\intrvol{k}$.  We set $\kappa_k:=\intrvol{k}(B_k)$ for $k\geq1$, and use the conventions $\kappa_0=1$ and $\intrvol{0}(\{p\})=1$.

We first isolate an elementary slicing fact used by the oracle.

\begin{lemma}[Quantiles and large affine sections]\label{lem:quantile-section}
Let $P\subseteq\tau B_m$ be a compact convex set of affine dimension $k\geq1$ and positive intrinsic volume $V:=\intrvol{k}(P)$.  Let $\ell$ be an affine functional that is nonconstant on $P$, and let $\alpha\in(0,1)$.  Then there is a level $r$ such that
\begin{equation}\label{eq:abstract-quantile}
 \intrvol{k}(P\cap\{\ell\geq r\})=\alpha V.
\end{equation}
If $u:=\max_P\ell$, then $r<u$, and some $y\in[r,u]$ satisfies
\begin{equation}\label{eq:abstract-large-section}
 \intrvol{k-1}(P\cap\{\ell=y\})
 \geq\frac{\alpha V}{2\tau}.
\end{equation}
\end{lemma}

\begin{proof}
Let $A:=\aff(P)$ and $L:=\lin(P-P)$.  Writing the linear part of $\ell$ as $b$, set $s:=\norm{\proj_L b}$.  Nonconstancy gives $s>0$.  The cap-volume function
\[
 h(c):=\intrvol{k}(P\cap\{\ell\geq c\})
\]
is continuous: after identifying $A$ isometrically with $\R^k$, every level set of the nonconstant affine functional has $k$-dimensional measure zero, so continuity follows from dominated convergence.  Moreover, $h(c)=V$ below $\min_P\ell$ and $h(c)=0$ at $c=u$.  The intermediate value theorem proves~\eqref{eq:abstract-quantile} and also $r<u$.

Choose orthonormal coordinates on $A$ whose first axis is $\proj_L b/s$.  Fubini's theorem, equivalently the coarea formula for this affine functional, gives
\begin{equation}\label{eq:coarea-general}
 \alpha V
 =\frac1s\int_r^u
 \intrvol{k-1}(P\cap\{\ell=y\})\,dy.
\end{equation}
For $p,q\in P$, the difference $p-q$ lies in $L$, and hence
\[
 \abs{\ell(p)-\ell(q)}
 \leq s\norm{p-q}
 \leq2\tau s.
\]
Thus $u-r\leq2\tau s$.  The average value of the integrand in~\eqref{eq:coarea-general} is therefore at least $\alpha V/(2\tau)$, proving~\eqref{eq:abstract-large-section}.
\end{proof}

\vspace{2em}
Suppose that the algorithm's $t$th query is
\[
 \q_t=(\x_t,\z_t)\in B_d.
\]
For each row define the affine functional
\begin{equation}\label{eq:row-functional}
 \ell_i(\w):=a x_{t,i}+\ip{\w}{\z_t}
 \qquad(\w\in P_i^{t-1}),
\end{equation}
and fix
\[
 \alpha:=\frac1{4m}.
\]
If $\ell_i$ is nonconstant on $P_i^{t-1}$, choose an upper $\alpha$-quantile $r_i$ as in~\eqref{eq:abstract-quantile}.  If $\ell_i$ is constant, let $r_i$ be its constant value.  Put
\begin{equation}\label{eq:max-quantile}
 r:=\max_{i\in[m]}r_i.
\end{equation}

\subsubsection{The update rule}\label{subsubsec:update}

We distinguish into two cases, informative versus uninformative updates, depending on whether the query reveals any new information about the true instance or not. 

\paragraph{Informative update.}
Suppose that some nonconstant row attains the maximum in~\eqref{eq:max-quantile}; choose the least such index $i$.  Apply Lemma~\ref{lem:quantile-section} to $P_i^{t-1}$ and $\ell_i$.  Since $\alpha=1/(4m)$, there is a level
\[
 y\in\left[r,\max_{\w\in P_i^{t-1}}\ell_i(\w)\right]
\]
for which
\begin{equation}\label{eq:large-section}
 \intrvol{k_i-1}(P_i^{t-1}\cap\{\ell_i=y\})
 \geq\frac{V_i}{8m\tau},
\end{equation}
where $k_i:=\dim P_i^{t-1}$ and $V_i:=\intrvol{k_i}(P_i^{t-1})$.  The adversary returns the exact value $y$ and sets
\begin{align}
 P_i^t&:=P_i^{t-1}\cap\{\ell_i=y\},
 \label{eq:selected-update}\\
 P_j^t&:=P_j^{t-1}\cap\{\ell_j\leq y\}
 \qquad(j\neq i).
 \label{eq:other-update}
\end{align}

\paragraph{Noninformative update.}
Suppose that every row attaining $r$ is constant.  The adversary returns $r$ and sets
\begin{equation}\label{eq:noninfo-update}
 P_j^t:=P_j^{t-1}\cap\{\ell_j\leq r\}
 \qquad(j\in[m]).
\end{equation}

The next proposition makes the exact-transcript invariant and the dimension and volume effects of these updates explicit.

\begin{proposition}[Exact product invariant and one-step bounds]\label{prop:product-invariant}
For every $t\geq0$, the sets $P_i^t$ are nonempty compact convex subsets of $\tau B_m$ with positive volume in their affine hulls.  They satisfy the simultaneous consistency invariant
\begin{equation}\label{eq:consistency-invariant}
 \text{every }W\in P_1^t\times\cdots\times P_m^t
 \text{ reproduces all responses through query }t\text{ exactly.}
\end{equation}
More precisely, write
\[
 k_j^-:=\dim P_j^{t-1},\quad
 V_j^-:=\intrvol{k_j^-}(P_j^{t-1}),\quad
 k_j^+:=\dim P_j^t,
\]
and let $V_j^+$ be the intrinsic volume of $P_j^t$ in its affine hull.

\begin{enumerate}[label=(\roman*),leftmargin=24pt]
\item In an informative update with selected row $i$,
\[
 k_i^+=k_i^--1,
 \qquad
 V_i^+\geq\frac{V_i^-}{8m\tau}.
\]
For every $j\neq i$,
\[
 k_j^+=k_j^-,
 \qquad
 V_j^+\geq(1-\alpha)V_j^-.
\]
If $\ell_j$ is constant, then in fact $P_j^t=P_j^{t-1}$.

\item In a noninformative update, for every $j$,
\[
 k_j^+=k_j^-,
 \qquad
 V_j^+\geq(1-\alpha)V_j^-.
\]
Every row on which $\ell_j$ is constant is unchanged.

\item If a query point has appeared previously, the response at its repetition is the same as before; the repeated query necessarily triggers a noninformative update and leaves every row set unchanged.
\end{enumerate}
\end{proposition}

\begin{proof}
We argue inductively.  The assertions hold at $t=0$.  All updates take closed convex subsets of compact convex sets, so it remains to prove nonemptiness, the stated volume and dimension bounds, and exact consistency.

Consider first an informative update.  The selected set has the positive $(k_i^--1)$-dimensional volume guaranteed by~\eqref{eq:large-section}.  Because $\ell_i$ is nonconstant on $\aff(P_i^{t-1})$, the affine hyperplane $\{\ell_i=y\}$ intersects that affine hull in dimension $k_i^--1$; positive intrinsic volume then implies that $P_i^t$ has exactly this affine dimension.  This proves the selected-row claims.

For an unselected nonconstant row $j$, one has $y\geq r\geq r_j$, and hence
\[
 P_j^{t-1}\cap\{\ell_j>y\}
 \subseteq
 P_j^{t-1}\cap\{\ell_j\geq r_j\}.
\]
The latter cap has volume $\alpha V_j^-$ by the quantile definition, so $P_j^t$ retains at least $(1-\alpha)V_j^->0$.  Positive $k_j^-$-dimensional volume inside the old affine hull forces $\dim P_j^t=k_j^-$.  If $\ell_j$ is constant, its value is $r_j\leq r\leq y$, so the entire row set is retained.

The selected row satisfies $\ell_i=y$, every other retained row satisfies $\ell_j\leq y$, and therefore every matrix in the new Cartesian product has
\[
 f_W(\q_t)=\max_{j\in[m]}\ell_j(\w_j)=y.
\]
All earlier responses remain exact because each new row set is contained in its predecessor.  This proves the invariant after an informative update.

For a noninformative update, at least one constant row has value $r$ and is retained, so it is an equality witness.  Every nonconstant row retains at least a $(1-\alpha)$ fraction of its intrinsic volume by the same cap argument, and every constant row has value at most $r$ and is unchanged.  Hence dimensions are preserved and every matrix in the updated product has value exactly $r$ at the current query.  The induction is complete.

It remains to justify the repeated-query assertion.  Suppose the same point previously received response $y$, and consider the product immediately before its repetition.  By~\eqref{eq:consistency-invariant},
\begin{equation}\label{eq:repeated-product-equality}
 \max_{i\in[m]}\ell_i(\w_i)=y
 \qquad\text{for every }(\w_1,\ldots,\w_m)
 \in\prod_iP_i^{t-1}.
\end{equation}
First, every row value is at most $y$: otherwise, choosing a point with value greater than $y$ in that row and arbitrary points in all other rows would contradict~\eqref{eq:repeated-product-equality}.  Second, at least one row is identically equal to $y$.  Indeed, if no row were identically $y$, one could choose in each row a point with value strictly below $y$, and the resulting product point would again contradict~\eqref{eq:repeated-product-equality}.

For a nonconstant row, the set $P_i^{t-1}\cap\{\ell_i=y\}$ lies in a proper affine hyperplane and has zero $k_i^-$-dimensional volume.  Its upper $\alpha$-quantile therefore satisfies $r_i<y$.  Constant rows have $r_i\leq y$, and at least one has $r_i=y$.  Thus $r=y$ and every row attaining $r$ is constant.  The update is noninformative, returns the same response $y$, and leaves every row set unchanged because all row values are at most $y$.
\end{proof}
\vspace{2em}

The construction is therefore valid for every deterministic query rule, however discontinuous, and the final Cartesian product consists entirely of fixed objectives consistent with the complete exact transcript.

\subsection{Dimension and volume behavior}\label{subsec:potential}

We now quantify how much uncertainty the update rule can remove.  Each informative update spends one unit of total row codimension, while every update causes only a controlled loss of intrinsic volume.  These two budgets imply that, below the desired query threshold, linearly many row sets retain both high dimension and large normalized volume radius.  Urysohn's inequality and an averaging argument then produce a single direction in which their aggregate width is large.

Let $T$ be the actual number of queries on the adversarial transcript, and let $s$ be the number of informative updates.  At the end, write
\[
 k_i:=\dim P_i^T,
 \qquad
 c_i:=m-k_i,
 \qquad
 V_i:=\intrvol{k_i}(P_i^T).
\]
By Proposition~\ref{prop:product-invariant}, every informative update lowers exactly one row dimension by one, and a noninformative update lowers none.  Hence
\begin{equation}\label{eq:codim-sum}
 \sum_{i=1}^m c_i=s\leq T.
\end{equation}

At an informative update, the product of current intrinsic row volumes is multiplied by at least
\[
 \frac1{8m\tau}(1-\alpha)^{m-1};
\]
at a noninformative update, it is multiplied by at least $(1-\alpha)^m$.  Bernoulli's inequality gives
\[
 \left(1-\frac1{4m}\right)^m\geq1-\frac14=\frac34.
\]
Starting from $P_i^0=\tau B_m$, we obtain
\begin{equation}\label{eq:product-volume}
 \prod_{i=1}^m V_i
 \geq
 (\kappa_m\tau^m)^m
 \left(\frac34\right)^T
 \left(\frac1{8m\tau}\right)^s.
\end{equation}

Normalize each final row body by
\begin{equation}\label{eq:rho-def}
 \rho_i:=\frac{V_i}{\kappa_{k_i}\tau^{k_i}}.
\end{equation}
Then $0<\rho_i\leq1$.  To see the upper bound, let $A_i:=\aff(P_i^T)$.  If $k_i=0$, the claim follows from our conventions.  If $k_i\geq1$, the section $A_i\cap\tau B_m$ is either empty or a $k_i$-dimensional Euclidean ball of radius
\[
 \sqrt{\tau^2-\dist(\zero,A_i)^2}\leq\tau.
\]
Since $P_i^T\subseteq A_i\cap\tau B_m$, its intrinsic volume is at most $\kappa_{k_i}\tau^{k_i}$.

We use the following elementary comparison of Euclidean ball volumes.

\begin{lemma}[Euclidean ball-volume ratio]\label{lem:kappa-ratio}
For integers $0\leq k\leq m$,
\[
 \frac{\kappa_m}{\kappa_k}\geq m^{-(m-k)/2}.
\]
\end{lemma}

\begin{proof}
For $1\leq j\leq m$, slicing the unit $j$-ball perpendicular to one coordinate gives
\[
 \frac{\kappa_j}{\kappa_{j-1}}
 =\int_{-1}^{1}(1-t^2)^{(j-1)/2}\,dt.
\]
For $j\geq2$, restrict the integral to $\abs{t}\leq1/\sqrt j$.  The elementary bound
\[
 \left(1-\frac1j\right)^{j-1}\geq e^{-1}>\frac14
\]
follows, for example, from $\log(1-x)\geq-x/(1-x)$ for $x\in(0,1)$.  It implies
\[
 \frac{\kappa_j}{\kappa_{j-1}}
 \geq\frac{2}{\sqrt j}
 \left(1-\frac1j\right)^{(j-1)/2}
 \geq\frac1{\sqrt j}
 \geq\frac1{\sqrt m}.
\]
For $j=1$, $\kappa_1/\kappa_0=2$.  Multiplication from $j=k+1$ through $m$ proves the result.
\end{proof}\vspace{2em}

Using $\sum_i k_i=m^2-s$, divide~\eqref{eq:product-volume} by $\prod_i\kappa_{k_i}\tau^{k_i}$.  The powers of $\tau$ cancel, and Lemma~\ref{lem:kappa-ratio} gives
\begin{align}
 \prod_{i=1}^m\rho_i
 &\geq
 \left(\frac34\right)^T
 \left(\frac1{8m}\right)^s
 \prod_{i=1}^m\frac{\kappa_m}{\kappa_{k_i}}
 \nonumber\\
 &\geq
 \left(\frac34\right)^T
 \left(\frac1{8m^{3/2}}\right)^s.
 \label{eq:rho-product}
\end{align}
Define
\[
 D_i:=-\log\rho_i,
 \qquad
 L_m:=\log\!\left(\frac{32}{3}m^{3/2}\right).
\]
Since $s\leq T$, equation~\eqref{eq:rho-product} implies
\begin{equation}\label{eq:entropy-budget}
 \sum_{i=1}^mD_i\leq TL_m
 \leq4T\log(em),
\end{equation}
where the last elementary inequality holds for every $m\geq1$.

\subsubsection{A common direction of large aggregate width}\label{subsubsec:width}

Fix
\[
 \eta:=\frac1{100}
\]
and suppose that the transcript length satisfies
\begin{equation}\label{eq:T-threshold}
 T\leq\eta\frac{m^2}{\log(em)}.
\end{equation}
We now show that many final row bodies remain simultaneously large in both affine dimension and normalized volume radius.

\begin{lemma}[Many good rows]\label{lem:good-rows}
Under~\eqref{eq:T-threshold}, there is a set $G\subseteq[m]$ with $\abs{G}\geq m/2$ such that, for every $i\in G$,
\begin{equation}\label{eq:good-row-bounds}
 k_i\geq(1-4\eta)m,
 \qquad
 \rho_i^{1/k_i}\geq e^{-1/6}>\frac12.
\end{equation}
\end{lemma}

\begin{proof}
By~\eqref{eq:codim-sum}, the number of rows satisfying $c_i>4\eta m$ is at most
\[
 \frac{T}{4\eta m}
 \leq\frac{m}{4\log(em)}
 \leq\frac m4.
\]
By~\eqref{eq:entropy-budget} and~\eqref{eq:T-threshold},
\[
 \sum_i D_i\leq4\eta m^2,
\]
so at most $m/4$ rows can satisfy $D_i>16\eta m$.  Remove both exceptional sets.  The remaining set $G$ has size at least $m/2$, and for every $i\in G$,
\[
 \frac{D_i}{k_i}
 \leq\frac{16\eta}{1-4\eta}
 =\frac16.
\]
Since $\rho_i^{1/k_i}=e^{-D_i/k_i}$, this proves~\eqref{eq:good-row-bounds}.
\end{proof}\vspace{2em}

For a compact convex set $P\subseteq\R^m$ and a direction $\thetaa\in\R^m$, define the full directional width
\[
 w_P(\thetaa)
 :=\max_{p\in P}\ip{\thetaa}{p}
 -\min_{p\in P}\ip{\thetaa}{p}.
\]
If $P$ has affine dimension $k\geq1$, let $L:=\lin(P-P)$ and let $\sigma_L$ be normalized spherical probability measure on the unit sphere $S(L)$.  With this full-width normalization, a Euclidean ball of radius $r$ has mean width $2r$.  Urysohn's inequality, applied after translating $P$ into $L$, states that
\begin{equation}\label{eq:urysohn}
 \int_{S(L)}w_P(\phi)\,d\sigma_L(\phi)
 \geq
 2\left(\frac{\intrvol{k}(P)}{\kappa_k}\right)^{1/k};
\end{equation}
see Schneider~\cite[Eq.~(7.21), p.~382]{Schneider2014}.

\begin{lemma}[A common wide direction]\label{lem:common-direction}
Under~\eqref{eq:T-threshold}, there is a unit vector $\thetaa\in\R^m$ such that
\begin{equation}\label{eq:aggregate-width}
 \sum_{i\in G}w_{P_i^T}(\thetaa)
 \geq\frac{m\tau}{4}.
\end{equation}
\end{lemma}

\begin{proof}
Fix $i\in G$ and put $L_i:=\lin(P_i^T-P_i^T)$.  By~\eqref{eq:urysohn},~\eqref{eq:rho-def}, and Lemma~\ref{lem:good-rows},
\begin{equation}\label{eq:intrinsic-width}
 \int_{S(L_i)}w_{P_i^T}(\phi)\,d\sigma_{L_i}(\phi)
 \geq2\tau\rho_i^{1/k_i}
 \geq\tau.
\end{equation}

Let $\thetaa$ be uniform on the ambient unit sphere $S^{m-1}$ and write
\[
 \proj_{L_i}\thetaa=R_i\phi_i,
 \qquad R_i\in[0,1],\quad \phi_i\in S(L_i).
\]
The standard Gaussian representation makes the distributional facts precise.  If $g$ is a standard Gaussian vector in $\R^m$ and $\thetaa=g/\norm{g}$, then $\proj_{L_i}g$ and $\proj_{L_i^\perp}g$ are independent isotropic Gaussian vectors.  On the probability-one event $\proj_{L_i}g\neq\zero$, the direction
\[
 \phi_i=\frac{\proj_{L_i}g}{\norm{\proj_{L_i}g}}
\]
is uniform on $S(L_i)$ and independent of
\[
 R_i=\frac{\norm{\proj_{L_i}g}}{\norm{g}}.
\]
Since width is positively homogeneous and depends only on the projection onto $L_i$,
\[
 w_{P_i^T}(\thetaa)=R_i w_{P_i^T}(\phi_i).
\]
Furthermore, $R_i\in[0,1]$ and rotational symmetry gives
\[
 \bb E R_i
 \geq\bb E R_i^2
 =\frac{k_i}{m}
 \geq1-4\eta
 \geq\frac12.
\]
Independence and~\eqref{eq:intrinsic-width} therefore imply
\[
 \bb E_{\thetaa}w_{P_i^T}(\thetaa)
 =\bb E R_i\,
   \bb E w_{P_i^T}(\phi_i)
 \geq\frac\tau2.
\]
Summing over $i\in G$ and using $\abs{G}\geq m/2$ gives
\[
 \bb E_{\thetaa}
 \sum_{i\in G}w_{P_i^T}(\thetaa)
 \geq\frac{m\tau}{4}.
\]
Some unit direction satisfies~\eqref{eq:aggregate-width}.
\end{proof}\vspace{2em}

\subsection{Two indistinguishable functions and completion of the proof}\label{subsec:separation}

We finally convert the residual row uncertainty into an optimization lower bound.  Taking opposite extremes of the good row sets in the common wide direction produces two fixed functions consistent with the same exact transcript.  The almost-uniform barycentric structure from Lemma~\ref{lem:barycentric} turns their aggregate row displacement into a constant separation between their unique minimizers.  The algorithm has one common output for both functions, so the growth inequality~\eqref{eq:quadratic-growth} forces error of order $d^{-1/2}$ on at least one of them.

Fix the direction $\thetaa$ from Lemma~\ref{lem:common-direction}.  Compactness of the row bodies allows us, for every $i\in G$, to choose
\begin{align*}
 \w_i^+&\in\argmaxop_{\w\in P_i^T}\ip{\thetaa}{\w},\\
 \w_i^-&\in\argminop_{\w\in P_i^T}\ip{\thetaa}{\w}.
\end{align*}
For $i\notin G$, choose any common point $\w_i^+=\w_i^-\in P_i^T$.  Let $W^+$ and $W^-$ be the resulting matrices.  Both belong to the final Cartesian product, so Proposition~\ref{prop:product-invariant} implies that $f_{W^+}$ and $f_{W^-}$ return the same exact value at every query in the transcript.

Equation~\eqref{eq:aggregate-width} gives
\begin{equation}\label{eq:mean-row-separation}
 \frac1m\sum_{i=1}^m
 \ip{\thetaa}{\w_i^+-\w_i^-}
 \geq\frac\tau4.
\end{equation}
Let
\[
 \p_{W^\pm}=(a\lambda^\pm,\z^\pm)
\]
be their closest points to the origin.

\begin{lemma}[Separation of the projection directions]\label{lem:projection-separation}
The normalized closest points satisfy
\begin{equation}\label{eq:u-separation}
 \norm{
 \frac{\p_{W^+}}{\norm{\p_{W^+}}}
 -
 \frac{\p_{W^-}}{\norm{\p_{W^-}}}
 }
 >\frac1{600}.
\end{equation}
\end{lemma}

\begin{proof}
Write
\[
 \lambda_i^\pm=\frac1m+e_i^\pm.
\]
By~\eqref{eq:lambda-close}, $\abs{e_i^\pm}\leq2/(\Gamma^2m)$.  In particular,
\[
 \abs{\sum_i e_i^\pm\ip{\thetaa}{\w_i^\pm}}
 \leq\sum_i\frac{2}{\Gamma^2m}\tau
 =\frac{2\tau}{\Gamma^2}.
\]
Using~\eqref{eq:mean-row-separation}, we obtain
\begin{align}
 \ip{\thetaa}{\z^+-\z^-}
 &=\frac1m\sum_i\ip{\thetaa}{\w_i^+-\w_i^-}
   +\sum_i e_i^+\ip{\thetaa}{\w_i^+}
   -\sum_i e_i^-\ip{\thetaa}{\w_i^-}
 \nonumber\\
 &\geq\frac\tau4-\frac{4\tau}{\Gamma^2}
 >\frac\tau5,
 \label{eq:z-separation}
\end{align}
where the last inequality uses $\Gamma=100$.  Hence
\begin{equation}\label{eq:p-separation}
 \norm{\p_{W^+}-\p_{W^-}}
 \geq\norm{\z^+-\z^-}
 >\frac\tau5
 =\frac{a}{5\Gamma\sqrt m}.
\end{equation}

Set
\[
 r_0:=\frac a{\sqrt m},
 \qquad
 R_0:=r_0\sqrt{1+\Gamma^{-2}},
 \qquad
 \widehat\p^\pm:=\frac{\p_{W^\pm}}{\norm{\p_{W^\pm}}}.
\]
By~\eqref{eq:p-norm-range}, $r_0\leq\norm{\p_{W^\pm}}\leq R_0$.  Decomposing the difference into an angular and a radial part gives
\begin{align*}
 \norm{\p_{W^+}-\p_{W^-}}
 &\leq R_0\norm{\widehat\p^+-\widehat\p^-}
   +\abs{\norm{\p_{W^+}}-\norm{\p_{W^-}}}\\
 &\leq R_0\norm{\widehat\p^+-\widehat\p^-}+(R_0-r_0).
\end{align*}
Combining this with~\eqref{eq:p-separation}, dividing by $r_0$, and using $\sqrt{1+t}\leq1+t/2$ for $t\geq0$ yields
\[
 \norm{\widehat\p^+-\widehat\p^-}
 >
 \frac{\frac1{5\Gamma}-\frac1{2\Gamma^2}}
      {1+\frac1{2\Gamma^2}}.
\]
At $\Gamma=100$, the right-hand side exceeds $1/600$.
\end{proof}\vspace{2em}

By~\eqref{eq:qstar}, the two unique minimizers are the negatives of the normalized closest points.  Lemma~\ref{lem:projection-separation} therefore separates the minimizers by more than $1/600$.

We now complete the proof of Theorem~\ref{thm:lower}.

Let
\[
 N_m:=\left\lfloor\eta\frac{m^2}{\log(em)}\right\rfloor
\]
and fix a deterministic algorithm that makes at most $N_m$ queries.  Run it against the resisting oracle defined in Subsection~\ref{subsec:resisting}.  If it stops early, let $T$ be its actual number of queries; then~\eqref{eq:T-threshold} holds.  Let $\widehat\q\in B_d$ be its output.

The two fixed functions $f_{W^+},f_{W^-}\in\cH_d$ constructed above reproduce the complete exact transcript.  Determinism therefore forces the algorithm to make the same queries and return the same output $\widehat\q$ on both functions.  Since their minimizers are more than $1/600$ apart, the triangle inequality gives, for at least one sign,
\[
 \norm{\widehat\q-\q_{W^\pm}^*}>\frac1{1200}.
\]
For that instance,~\eqref{eq:quadratic-growth} and the lower bound in~\eqref{eq:p-norm-range} imply
\begin{align*}
 f_{W^\pm}(\widehat\q)-\min_{B_d}f_{W^\pm}
 &\geq
 \frac{a}{2\sqrt m}\left(\frac1{1200}\right)^2\\
 &=\frac{1}{5{,}760{,}000\sqrt m}\\
 &>\frac{10^{-7}}{\sqrt d},
\end{align*}
where the last inequality uses $m\leq d$ and $1/5{,}760{,}000>10^{-7}$.  Thus every algorithm using at most $N_m$ queries fails on some member of $\cH_d$ at accuracy $10^{-7}/\sqrt d$.

Because $m=d/2$ and $\log(em)\leq2\log(d+1)$ for all sufficiently large even $d$,
\[
 \eta\frac{m^2}{\log(em)}
 \geq\frac{\eta}{8}\frac{d^2}{\log(d+1)}.
\]
Every integer query budget through $N_m$ fails, so the successful worst-case complexity is at least $N_m+1>\eta m^2/\log(em)$.  Taking $c=\eta/8$ and increasing $d_0$ if necessary proves Theorem~\ref{thm:lower} in even dimension; the reduction at the start of the section gives the odd-dimensional case.

\section{Lifting to the Mixed-Integer Case}\label{sec:mixed-transfer}

We first specify the mixed-integer value-oracle model used in this section.  For $n,d\geq1$, let
\[
 \mathcal D_{n,d}:=[0,1]^n\times[-1,1]^d,
 \qquad
 \mathcal D_{n,d}^{\mathrm{MI}}:=\{0,1\}^n\times[-1,1]^d.
\]
Let $\cF_{n,d}^{\mathrm{MI}}$ be the class of finite convex functions $F:\mathcal D_{n,d}\to\R$ such that $F(\x,\cdot)$ is Euclidean $1$-Lipschitz on $[-1,1]^d$ for every $\x\in\{0,1\}^n$.  A deterministic algorithm may query the exact value of $F$ at any point of $\mathcal D_{n,d}$, including points with fractional first coordinates, and must output $(\widehat\x,\widehat\y)\in\mathcal D_{n,d}^{\mathrm{MI}}$.  Write $Q_{\mathrm{MI},\val}^{\det}(n,d,\eps)$ for the least worst-case number of queries needed to guarantee
\[
 F(\widehat\x,\widehat\y)
 -\min_{(\x,\y)\in\mathcal D_{n,d}^{\mathrm{MI}}}F(\x,\y)
 \leq\eps
 \qquad(F\in\cF_{n,d}^{\mathrm{MI}}).
\]
The product domain is known, so no feasibility oracle is part of this model.
With unrestricted deterministic computation, this query complexity agrees with the information-complexity definition used below: a transcript permits a valid output exactly when all instances consistent with it have a common $\eps$-approximate solution.

We use the following transfer result.  Its terminology is recalled only to make clear the hypotheses needed below.

\begin{theorem}[Basu--Jiang--Kerger--Molinaro transfer theorem]\label{thm:bjkm-transfer}
In the notation of Basu, Jiang, Kerger, and Molinaro~\cite[Theorem~7]{BasuJiangKergerMolinaro2025}, let $\{\mathcal H_{n,d}\}$ be a hereditary family of permissible queries such that $\mathcal H_{0,d}$ contains every value-threshold query
\[
 h_c(v)=\operatorname{sgn}(v+c),\qquad c\in\R.
\]
Suppose that $\mathcal I\subseteq\mathcal I_{0,d,R,\rho,M}$ is a class of continuous convex unconstrained instances, where ``unconstrained'' means that the feasible region is the full box $[-R,R]^d$, and that $\mathcal G_0$ is a first-order chart satisfying
\[
 \operatorname{icomp}_{\eps}
 \bigl(\mathcal I,\mathcal O(\mathcal G_0,\mathcal H_{0,d})\bigr)
 \geq\ell.
\]
If all instances in $\mathcal I$ have the same optimal value, then, for every $n\geq1$, there is a first-order chart $\mathcal G_n$ such that
\[
 \operatorname{icomp}_{\eps}
 \bigl(\mathcal I_{n,d,R,\rho,M},
       \mathcal O(\mathcal G_n,\mathcal H_{n,d})\bigr)
 \geq2^{n-1}\ell.
\]
\end{theorem}

For exact values, the oracle condition in Theorem~\ref{thm:bjkm-transfer} causes no loss.  Take the permissible value queries to be the identity together with all $h_c$, and allow no subgradient or separation queries.  This family is hereditary: under a known shift $v\mapsto v+\delta$, an identity response is postprocessed by adding $\delta$, while $h_c(v+\delta)=h_{c+\delta}(v)$.  Conversely, one exact value answers any threshold query with free postprocessing.  Hence this augmented oracle and the ordinary exact-value oracle have identical query complexity; the first-order chart in the transfer theorem is immaterial.

\begin{proof}[Proof of Corollary~\ref{cor:mixed-value}]
We first prepare the continuous hard family for Theorem~\ref{thm:bjkm-transfer}.  Assume initially that $d=2m$ and put $\eps_d:=\eps_0/\sqrt d$.  By~\eqref{eq:min-value} and~\eqref{eq:p-norm-range}, the optimal values of the functions in $\cH_d$ lie in an interval of length
\begin{equation}\label{eq:optimum-spread}
 \Delta_d
 =\frac a{\sqrt m}\left(\sqrt{1+\Gamma^{-2}}-1\right)
 =\frac{a\sqrt2}{\sqrt d}\left(\sqrt{1+\Gamma^{-2}}-1\right).
\end{equation}
Thus this interval can be partitioned into a universal number
\[
 K:=\left\lceil
 \frac{2a\sqrt2}{\eps_0}
 \left(\sqrt{1+\Gamma^{-2}}-1\right)
 \right\rceil
\]
of intervals of length at most $\eps_d/2$.  With the constants used in the proof, $K=708$.

At least one of these $K$ subclasses retains the lower bound $\Omega(d^2/\log(d+1))$ at accuracy $\eps_d$.  Indeed, let $T_j$ be the complexity of subclass $j$.  Run one optimal algorithm for each subclass, evaluate the $K$ returned points, and return the one having the smallest value.  The algorithm belonging to the unknown function's own subclass produces an $\eps_d$-solution, so this procedure solves all of $\cH_d$ using at most $\sum_{j=1}^K T_j+K$ queries.  The lower bound proved above therefore gives
\[
 \max_{j\in[K]}T_j
 \geq\frac{N_m+1-K}{K}
 =\Omega\!\left(\frac{d^2}{\log(d+1)}\right),
\]
where the last equality holds for all sufficiently large $d$ because $K$ is universal.

Fix such a subclass, and let $[\alpha_d,\beta_d]$ be its interval of optimal values.  For $W$ in this subclass and $\q\in\R^d$, define the global convex functions
\begin{equation}\label{eq:box-floor-transform}
 \Phi_W(\q):=f_W(\q)+(\norm{\q}-1)_+,
 \qquad
 g_W(\q):=\frac12\max\{\Phi_W(\q),\beta_d\}.
\end{equation}
Here $s_+:=\max\{s,0\}$.  These functions are convex.  The first term in $\Phi_W$ is Euclidean $1$-Lipschitz and the radial penalty is Euclidean $1$-Lipschitz, so every $g_W$ is Euclidean $1$-Lipschitz.

This transformation loses only a constant factor in accuracy and no queries.  To see this, set
\[
 r:=\max\{1,\norm{\q}\},
 \qquad
 \q^\circ:=\frac{\q}{r}\in B_d.
\]
The max-linear function $f_W$ is positively homogeneous, and hence
\begin{equation}\label{eq:radial-query-simulation}
 \Phi_W(\q)=r f_W(\q^\circ)+(r-1).
\end{equation}
Thus one exact query to $f_W$ simulates one exact query to $g_W$.  Moreover, $f_W(\q^\circ)\geq-1$, so~\eqref{eq:radial-query-simulation} gives
\[
 \Phi_W(\q)\geq f_W(\q^\circ).
\]
It follows that $\min_{[-1,1]^d}\Phi_W=\min_{B_d}f_W$.  Because this minimum belongs to $[\alpha_d,\beta_d]$, every $g_W$ has the same optimal value $\beta_d/2$.

Now suppose that $\widehat\q$ is an $\eps_d/4$-solution of $g_W$.  Then
\[
 \Phi_W(\widehat\q)\leq\beta_d+\frac{\eps_d}{2}.
\]
Radially projecting $\widehat\q$ and using the width of the selected interval gives
\begin{align*}
 f_W(\widehat\q^{\circ})-\min_{B_d}f_W
 &\leq
 \left(\beta_d-\min_{B_d}f_W\right)+\frac{\eps_d}{2}\\
 &\leq\eps_d.
\end{align*}
Consequently, the unconstrained box family $\{g_W\}$, whose members have a common optimal value, still requires $\Omega(d^2/\log(d+1))$ exact-value queries at accuracy $\eps_d/4$.  It satisfies the parameters of Theorem~\ref{thm:bjkm-transfer}: the feasible region is $[-1,1]^d$, one may take $R=\rho=1$, and Euclidean $1$-Lipschitz continuity implies $\sqrt d$-Lipschitz continuity in the $\ell_\infty$ norm.

Applying Theorem~\ref{thm:bjkm-transfer} with the exact-value specialization described above gives
\[
 Q_{\mathrm{MI},\val}^{\det}\!\left(n,d,\frac{\eps_0}{4\sqrt d}\right)
 =\Omega\!\left(2^n\frac{d^2}{\log(d+1)}\right).
\]
The construction in the proof of the transfer theorem~\cite[Section~2]{BasuJiangKergerMolinaro2025} uses the known product domain $[0,1]^n\times[-1,1]^d$, the $2^n$ fibers indexed by $\{0,1\}^n$, and queries at fractional first coordinates.  On each binary fiber it uses a member of the common-optimum family above or a finite pointwise maximum of such members, so the fiber remains Euclidean $1$-Lipschitz.  Hence its hard instances belong to $\cF_{n,d}^{\mathrm{MI}}$.  For odd $d$, let $d'=d-1$ and make the objective independent of the last continuous coordinate.  Since $\eps_0/(4\sqrt d)\leq\eps_0/(4\sqrt{d'})$ and $d'^2/\log(d'+1)=\Omega(d^2/\log(d+1))$, monotonicity gives the same result after changing only the universal constants.  This proves the lower bound with, for example, $\eps_{\mathrm{MI}}=\eps_0/4$.

For the upper bound, enumerate the $2^n$ binary fibers.  On each one, apply the Protasov exact-value bound quoted before~\eqref{eq:protasov-upper} to the convex function $F(\x,\cdot)$ on $[-1,1]^d$.  Its objective range is at most $2\sqrt d$, so at accuracy $\eps_{\mathrm{MI}}/\sqrt d$ this takes $O(d^2\log^2(d+1))$ queries per fiber.  Evaluate the returned candidates and select the one with smallest value.  The candidate from an optimal fiber is accurate for the mixed-integer problem, and the selection can only improve its value.  The total is $O(2^n d^2\log^2(d+1))$, completing the proof.
\end{proof}\vspace{2em}

\section{Discussion}\label{sec:discussion}

The main implication of our result is that exact values and full first-order information have polynomially different deterministic complexities on the same nonsmooth convex class.  This is not a finite-precision effect: the algorithm receives exact real values and may use unlimited computation and memory.  The separation therefore isolates the cost of receiving no subgradient or separating direction.

The proof's central feature is the exact Cartesian-product uncertainty invariant.  One row of the max-affine objective certifies each response while the other rows retain enough dimension and intrinsic volume.  Aggregate width and the almost-uniform barycentric structure then turn this residual row uncertainty into separated minimizers.  This mechanism avoids assigning a finite information content to an exact real response and may be useful for other partial-information oracles.

The mixed-integer consequence shows that the continuous difficulty combines multiplicatively with the integer dimension: the complexity is $\widetildeTheta(2^n d^2)$ even when queries with fractional first coordinates are allowed.  Its proof also shows that applying a mixed-integer transfer theorem may require nontrivial preparation of the continuous hard family, here through the common-optimum and full-box reductions.

Several questions remain open.  The lower and upper bounds still differ by logarithmic factors, and it is unclear whether the losses in the volume argument or in Protasov's method are intrinsic.  The present hard family has optimum depth $\Theta(d^{-1/2})$, so a different construction would be needed to obtain a comparable lower bound at dimension-free accuracy.  The resisting oracle is tied to a deterministic transcript and does not yield a hard distribution for randomized algorithms.  Finally, the result does not address whether a superlinear value-only lower bound persists under additional regularity such as smoothness or strong convexity.

\backmatter

\bmhead{Acknowledgments}
As noted in the introduction, modern AI tools were used extensively to establish the presented results. See appendix \ref{subsec:ai-usage} for details. 

\bibliography{value_oracle_references}

\newpage
\appendix
\section{AI usage and methodology}\label{appendix}

Section \ref{prompt for initial proof} gives the full prompt used that led to the initial proof after 148 minutes of processing. This initial proof was refined into what is in this manuscript. The exact chat, including the full response, for the initial proof can be viewed \href{https://chatgpt.com/share/6a55aa50-b484-83ea-85c0-c7e7b4bda41c}{here}: \begin{verbatim}
    https://chatgpt.com/share/6a55aa50-b484-83ea-85c0-c7e7b4bda41c
\end{verbatim}
Note that this proof actually works for accuracy of order $d^{-3}$, not just the stated $d^{-4}$. This is the result that we have formally verified in Lean (the improvement to $d^{-1/2}$ is a much larger undertaking to formally verify in Lean, largely due to absence in avalable Lean libraries of much advanced convex geometry needed for Urysohn's inequality). 

The chat that led to the main refinement of the accuracy, using a similar prompt in addition the initial proof, followed by 230 minutes of thinking on GPT Sol 5.6 Pro, can be found \href{https://chatgpt.com/share/6a55ad10-7644-83ea-859e-5483d2e0dff0}{here}: 
\begin{verbatim}
    https://chatgpt.com/share/6a55ad10-7644-83ea-859e-5483d2e0dff0
\end{verbatim}

\subsection{Initial Prompt}\label{prompt for initial proof}
Current task statement:

For every integer $d \ge 2$, let

\[B_d = \{x \in \mathbb{R}^d : \lVert x\rVert_2 \le 1\]

be the closed Euclidean unit ball. Let $F_d$ be the class of all real-valued functions $f : B_d \to \mathbb{R}$ satisfying:

$f$ is convex on $B_d$;

$f$ is 1-Lipschitz in the Euclidean norm:$|f(x) - f(y)| \le \lVert x-y\rVert_2$for all $x,y$ in $B_d$;

$f(0) = 0$.

No differentiability, smoothness, strict convexity, strong convexity, polyhedrality, finite representation, or other structure is assumed. The class contains both smooth and nonsmooth functions; “nonsmooth convex optimization” here means that an algorithm must work without any smoothness assumption.

The only access to $f$ is an exact function-value oracle. When queried at $x$ in $B_d$, the oracle returns the exact real number $f(x)$. It returns no gradient, subgradient, separating hyperplane, proximal point, comparison bit, stochastic estimate, or other information.

A deterministic zeroth-order algorithm may adaptively choose queries

\[x_1, x_2, \ldots, x_T \in B_d,\]

where $x_t$ may be an arbitrary deterministic function of $d$ and the exact previous answers $f(x_1),\ldots,f(x_{t-1})$. After at most $T$ oracle calls it outputs an arbitrary point $x_{\mathrm{hat}}$ in $B_d$. The output need not have been queried. The algorithm may perform unlimited computation and exact real arithmetic between queries. Its query maps maybe discontinuous. There is no running-time, memory, numerical-stability, linear-span,locality, or finite-precision restriction. The algorithm and all of its rules must be fixed independently of the unknown $f$, and it has no access to $f$ except through the value oracle.

Queries outside $B_d$ are not permitted. Repeated queries count separately. Queries issued in a batch or in parallel count individually. Every evaluation used in a linesearch, interpolation procedure, rejected trial, gradient estimate, initialization step, or stopping test counts as an oracle call. Computations involving the explicitly known ball $B_d$ are free because only function-value oracle complexity is being measured.

Define $Q_{\mathrm{val}}^{\mathrm{det}}(d,\epsilon)$ to be the smallest integer $T$ for which there exists such a deterministic exact-value algorithm satisfying

\[f(x_{\mathrm{hat}}) - \min_{x \in B_d} f(x) \le \epsilon\]

for every $f$ in $F_d$.

Resolve completely the polynomial dependence on $d$ of

\[Q_{\mathrm{val}}^{\mathrm{det}}(d,d^{-4}).\]

A complete resolution must identify an explicit exponent $\alpha$ and prove

\[Q_{\mathrm{val}}^{\mathrm{det}}(d,d^{-4}) = \widetilde{\Theta}(d^\alpha),\]

or, if the answer is not a pure power of $d$, identify an explicit function $q(d)$ and prove

\[Q_{\mathrm{val}}^{\mathrm{det}}(d,d^{-4}) = \widetilde{\Theta}(q(d)).\]

Here $\widetilde{O}$, $\widetilde{\Omega}$, and $\widetilde{\Theta}$ may hide only fixed powers of log $d$. Since $\epsilon=d^{-4}$, logarithmic dependence on $1/\epsilon$ is also logarithmic in $d$. These symbols may not conceal an unspecified polynomial factor, an arbitrary $d^{o(1)}$ factor, or a change of oracle model. A characterization up to universal constant factors is also acceptable and is stronger than required.

Equivalently, it is acceptable to establish matching upper and lower bounds whose ratio is polylogarithmic in $d$. If the actual answer has a more complicated explicit asymptotic form than $d^\alpha$, prove that form rather than forcing it into a power law.

Assume for purposes of this task that a complete resolution in the exact model above is obtainable. Do not terminate merely by saying that the problem is open or that existing techniques leave a gap.

Verified baseline and literature checkpoint

The following facts motivate the task, but every use of them must be checked against the exact model above.

Protasov (1996), “Algorithms for approximate calculation of the minimum of a convexfunction from its values,” gives a deterministic value-only upper bound recorded in later literature as

\[O(d^2 \log d \log(1/\epsilon))\]

evaluations for an explicit convex domain. At $\epsilon=d^{-4}$, this is$\widetilde{O}(d^2)$.

Woodworth and Srebro (2019), “Open Problem: The Oracle Complexity of Convex Optimization with Limited Memory,” formulated a memory-query tradeoff problem for first-order convex optimization.

Marsden, Sharan, Sidford, and Valiant (2022), “Efficient Convex Optimization Requires Superlinear Memory,” prove a first-order tradeoff of the form

\[\mathrm{memory} \le d^{5/4-\delta}\ \mathrm{bits}\]
implies
\[\mathrm{queries} \ge \widetilde{\Omega}(d^{1+4\delta/3})\]

over an appropriate inverse-polynomial accuracy regime.

Blanchard, Zhang, and Jaillet (2023), “Quadratic Memory is Necessary for Optimal Query Complexity in Convex Optimization: Center-of-Mass is Pareto-Optimal,” prove that, at accuracy $1/d^4$, a deterministic first-order algorithm using at most

\[d^{2-\delta}\ \mathrm{bits}\]

requires

\[\widetilde{\Omega}(d^{1+\delta/3})\]

queries, for $\delta \in [0,1]$.

A frequently suggested argument stores the transcript of a $T$-query value-only algorithm and applies one of these memory lower bounds. If one could always store or replay the transcript in $\widetilde{O}(T)$ bits, the Marsden et al. tradeoff would suggest an exponent $8/7$, while the Blanchard-Zhang-Jaillet tradeoff would suggest an exponent $5/4$.

This inference is not automatic in the present task. Each value-oracle answer is an exact real number and can have unbounded bit complexity. Moreover, future queries may depend discontinuously on previous exact values. No solution may assume that one scalar answer occupies $O(\log d)$ bits, that all replies can harmlessly be rounded, or that a $T$-query transcript has $\widetilde{O}(T)$ bits. Any use of a memory-constrained first-order theorem must include a complete, model-preserving reduction or a new hard family whose exact replies provably carry only the claimed amount of usable information.

A complete solution must prove exactly the following

There must be a matching algorithmic upper bound and information-theoretic lower bound for $Q_{\mathrm{val}}^{\mathrm{det}}(d,d^{-4})$, up to polylogarithmic factors.

For the upper bound, give an explicit deterministic algorithm and prove that, for every $f$ in $F_d$:

every oracle query lies in $B_d$;

only exact function values are used;

the algorithm terminates after the claimed worst-case number of evaluations;

its output lies in $B_d$;

its output satisfies $f(x_{\mathrm{hat}}) - \min_{B_d} f \le d^{-4}$;

every auxiliary function evaluation is included in the count.

For the lower bound, prove that every deterministic adaptive exact-value algorithm using fewer than the claimed number of evaluations fails on at least one fixed function $f$ in $F_d$. The lower bound must apply even when the algorithm:

performs unlimited computation and exact real arithmetic;

uses arbitrary and discontinuous decision rules;

retains its complete exact transcript;

chooses arbitrary adaptive query points;

outputs a point that was never queried;

knows $d$, the unit ball, the Lipschitz constant, the normalization $f(0)=0$, and the target accuracy in advance.

A hard-instance proof may use a special subclass of $F_d$, such as polyhedral functions,support functions, maxima of affine functions, gauges, or distance functions, because a lower bound on a subclass is a valid lower bound for the full class. However, every hard function produced must actually be convex, 1-Lipschitz on $B_d$, normalized by $f(0)=0$, and consistent with all exact answers supplied during the interaction.

Closing the gap may take any mathematically correct form. For example:

a lower bound matching the existing $\widetilde{O}(d^2)$ upper bound;

a deterministic algorithm with exponent strictly below 2 together with a matching lower bound;

matching upper and lower bounds at an intermediate exponent;

a sharper non-power-law characterization.

A strict improvement on only one side does not count as a complete resolution.

Results that do not count

The following are insufficient unless they are accompanied by arguments that imply the complete matching result above.

An improved lower bound such as $d^{5/4}$, $d^{4/3}$, $d^{3/2}$, or $d^{2-o(1)}$ with no matching algorithmic upper bound.

An algorithm using $o(d^2)$ evaluations with no matching lower bound.

A randomized algorithm, including one that succeeds with high probability, in expectation, or after fixing a favorable random seed.

A randomized lower bound that is not correctly converted into a worst-case lower bound for every deterministic algorithm.

An algorithm or lower bound for a noisy, stochastic, rounded, $b$-bit, comparison-only, sign, order, bandit-feedback, or approximate-value oracle.

A finite-precision result without a rigorous lifting theorem to the exact-value oracle specified here.

A first-order, subgradient, separation, cutting-plane-oracle, membership-oracle,proximal-oracle, linear-minimization-oracle, or higher-order result without a complete reduction using only counted function evaluations.

A result restricted to linear-span algorithms, continuous algorithms, comparison-based algorithms, local algorithms, finite-difference algorithms, bounded-memory algorithms, algebraic decision trees of bounded degree, or any other proper subclass of deterministic exact-value algorithms.

An upper bound restricted to smooth, strongly convex, separable, symmetric, coordinatewise, polyhedral-with-few-facets, or otherwise specially structured objectives.

A lower bound that assumes the output must be one of the queried points.

A lower bound that treats a real oracle answer as one bit or one machine word.

A transcript-counting or entropy argument that has only finitely many possible answers after silently rounding exact values.

A proof that a rounded transcript is close to the exact transcript without proving that all later adaptive queries and the final guarantee remain valid under that rounding.

A resisting-oracle argument whose adaptive answers do not extend, after the interaction, to one fixed function $f$ in $F_d$.

A construction that is convex only on the queried points but has no valid global convex 1-Lipschitz extension to $B_d$.

A construction whose affine pieces have slopes of norm greater than 1, whose normalization is inconsistent, or whose claimed optimum is outside $B_d$.

A lower bound at constant accuracy, at an unspecified $1/\operatorname{poly}(d)$ accuracy, or at an accuracy not shown to imply the stated $d^{-4}$ result.

An algorithm whose query bound hides polynomial dependence on $1/\epsilon$ that becomes an additional power of $d$ when $\epsilon=d^{-4}$.

A reduction to another unresolved conjecture, an unproved “stability lemma,” or a global compatibility statement equivalent in strength to the original problem.

Computational verification through any fixed dimension.

Numerical evidence, performance-estimation experiments, symbolic searches, or candidate hard instances without a proof for arbitrary $d$.

A faster arithmetic implementation of an existing $d^2$-query method without reducing the number of function evaluations.

A literature survey or a statement that no known method closes the gap.

Use multiagent v2 aggressively and dynamically

Use up to $64$ concurrent agents. Do not use a fixed allocation such as “$N$ agents for lower bounds and $N$ agents for algorithms.” Continually create, merge, redirect, pause, and terminate agent lines according to their mathematical progress.

Begin with a genuinely diverse portfolio. Preserve substantial independence during early rounds: do not tell most agents which route currently looks most promising, and do not let one attractive reduction cause premature convergence.

Maintain an explicit registry of approach families. Group agents by their actual mathematical mechanism, not by superficial wording. For every family, record:

its central proposed lemma or construction;

what has been proved rigorously;

its exact remaining gap;

whether that gap is local or theorem-strength;

known counterexamples or failed variants;

whether the route is active, blocked, merged, or abandoned.

Initial approach families should include substantially different representatives ofat least the following kinds.

Resisting oracles and convex interpolation

Study adaptive value assignments that remain extendable to a convex 1-Lipschitzfunction. Seek exact interpolation or extension criteria for finite value data. Investigate max-of-affine extensions, convex envelopes, and mechanisms for retaining a large set of possible minimizers after many exact queries.

Do not assume that locally plausible answers have a common global extension. Produce the extension explicitly or prove a necessary-and-sufficient interpolation theorem.

Convex geometry and epigraph formulations

View a value query as information about the boundary of the epigraph. Explore polarity, support functions, gauges, hidden convex bodies, localization sets, volumetric arguments, widths, centroid methods, and duality between function minimization and geometric reconstruction.

Determine precisely how much geometric information one exact height measurement can supply. Do not replace a value query by a separation or membership query unless the reduction is proved in the correct direction with the correct number of calls.

Information, communication, and memory arguments

Explore reductions from communication problems, streaming problems, orthogonal-vector games, hidden subspaces, indexing, or memory-query tradeoffs. Design hard families whose exact function values are provably information-controlled even for arbitrary real queries.

Any route based on storing a value transcript must confront exact real precision directly. A useful lemma must explain why exact replies cannot encode or reveal the hidden instance too quickly, or why replies can be quantized without changing the algorithm’s behavior or success guarantee.

Topological and decision-theoretic lower bounds

Explore continuous and discontinuous decision trees separately, adversary dimensions, Borsuk-Ulam-type mechanisms, invariance-of-domain ideas, widths of function classes, topological complexity, and algebraic or measure-theoretic indistinguishability.

A lower bound for continuous decision rules does not apply to this task unless it is extended to arbitrary decision rules. Clearly identify every regularity assumption.

Polyhedral and support-function hard instances

Investigate maxima of many affine forms, distance-to-hidden-set functions, support functions of hidden bodies, nested ridges, multiscale facets, and adversarial arrangements. Exact values must not inadvertently identify all hidden facets or encode the minimizer in one scalar.

Require concrete formulas, norm bounds, optimum calculations, and transcript consistency proofs.

Value-to-separation and value-to-subgradient upper bounds

Search for deterministic procedures that extract useful separating information from fewer than $\Theta(d)$ new evaluations per localization step. Investigate approximate subgradients, simplex gradients, directional derivatives, asymmetric differences near the boundary, adaptive interpolation, and weak separation oracles.

Every approximate separator must have a proved error bound strong enough for the global optimization method. Do not call an approximation “a separator” without proving that it retains the true minimizer or yields the claimed objective guarantee.

Amortization and sample reuse

The standard $d^2$ pattern can be viewed as roughly $d$ localization stages with roughly $d$ new evaluations per stage. Explore whether values gathered in earlier stages can be reused, whether several cuts can be obtained from one interpolation model, whether multiscale sampling can amortize directional information, and whether the number of fresh probes per cut can decline as the localization region shrinks.

Account for all conditioning errors and for nonsmooth behavior. A heuristic that works for a stable gradient field is not sufficient for arbitrary convex Lipschitz functions.

Duality, conjugacy, and minimax formulations

Explore Fenchel conjugates, saddle formulations, support-function duality, bundle models, cutting-plane dual certificates, and minimax exchanges. Determine whether exact values provide indirect access to a dual object that can be optimized more efficiently.

Do not assume access to the conjugate, a maximizing subgradient, or a dual separation oracle unless it is constructed from counted value queries.

Multiscale localization and extremal arguments

Study nested localization, potential functions, center-of-gravity or volumetric progress, shallow cuts, ellipsoidal geometry, and extremal configurations of queried points. Seek either an invariant forcing $\Omega(d)$ information cost per effective cut or an algorithm that obtains global progress without explicitly reconstructing a cut.

Computational and symbolic sanity checks

Use low-dimensional linear programs, convex-extension feasibility programs, adversarial searches, and exact symbolic calculations to falsify proposed lemmas quickly. These experiments may guide proofs but never count as the final result.

Agents must return concrete mathematics

Require every agent to return at least one of the following:

a formally stated lemma with proof;

an explicit algorithm and proved invariant;

an explicit hard family and oracle transcript rule;

a complete reduction with all parameter transformations;

a concrete counterexample to a proposed lemma;

exact equations showing why a claimed exponent follows;

a finite-dimensional optimization program whose solution rigorously certifies a local claim.

Reject status reports, broad suggestions, vague optimism, and statements that a precision, extension, or compatibility step is “standard” or “routine.”

Do not let an approach dominate merely because it yields an elegant reduction. A route ending at a lemma equivalent in strength to the original gap is not close to completion. When a route stalls at a theorem-strength missing lemma, mark it blocked. Reopen it only when an agent supplies a materially new mechanism, invariant, construction, or counterexample that addresses the recorded obstruction.

Keep several incompatible routes alive through multiple rounds. Cross-pollinate only after independent agents have developed their routes far enough that their true strengths and gaps are visible.

Required adversarial audit for every lower bound

Assign independent adversarial agents to check all of the following.

Exact-oracle fidelity

Does the proof apply to exact real-valued answers, or only to rounded, noisy, or finite-alphabet replies?

Infinite-information loopholes

Can a reply’s low-order digits encode the hidden instance? Can the algorithm choose a query designed to decode such information? Does the proof improperly count scalar answers rather than bits?

Arbitrary adaptivity

Does the lower bound cover discontinuous query maps and exact comparisons, or only stable or continuous algorithms?

Fixed-function consistency

After the complete adaptive interaction, is there one fixed $f$ in $F_d$ producing every answer exactly? A sequence of answers chosen from different functions is invalid.

Global convex extension

Are all transcript values jointly extendable to a convex function on the entire ball? Is the extension real-valued and 1-Lipschitz?

Normalization

Does the final hard function satisfy $f(0)=0$ without altering the transcript or leaking the hidden instance?

Lipschitz constant

For max-of-affine constructions, does every slope have Euclidean norm at most 1? For other constructions, is the global Lipschitz bound proved?

Output handling

Does the proof defeat an arbitrary output point, including one never queried? It is permissible to append the output as one extra query if the reduction explicitly justifies why this changes the complexity by only one.

Optimization gap

Is the final output gap strictly greater than $d^{-4}$, with all scaling and constants verified?

Domain fidelity

Are all algorithm queries allowed by the model, and is the adversary defined for every possible query in $B_d$?

Deterministic universality

Does the lower bound quantify over every deterministic algorithm, rather than a restricted algorithmic family?

Quantifier order

The proof must establish: for every algorithm using fewer than $T$ queries, there exists a fixed $f$ in $F_d$ on which it fails. Check that randomness used to sample a hard instance, if any, is converted correctly into such a deterministic worst-case statement.

Imported-theorem alignment

Do all imported memory, communication, interpolation, or convex-geometry theorems have exactly the needed hypotheses? Check dimension, norm, accuracy, oracle response, randomness, output, memory unit, and query-domain conventions.

Noncircularity

Does the main missing lemma secretly assert that value queries cannot solve the optimization problem faster, or otherwise restate the desired lower bound?

Required adversarial audit for every upper bound

Assign separate agents to check all of the following.

General nonsmoothness

Does the algorithm work at kinks, on flat faces, and when every queried point relevant to the construction is nondifferentiable?

Boundary behavior

Are all finite-difference, interpolation, or line-search points inside $B_d$, including when the current center is near the boundary?

Separator validity

If the algorithm constructs an approximate subgradient or cut, is its error proved for all convex 1-Lipschitz functions? Does the retained localization region still contain an optimizer or a $d^{-4}$-optimal point?

Accumulated error

Do approximation errors remain controlled over all stages, rather than being analyzed for one isolated step?

No hidden oracle

Does every claimed gradient, support plane, membership answer, or comparison arise solely from explicitly counted function evaluations?

Unknown optimum

Does the method avoid assuming knowledge of min $f$, an optimal point, or a valid initial lower bound not derivable from the model?

Query accounting

Are initialization, calibration, binary search, rejected probes, repeated evaluations, and final certification all included?

Worst-case rather than generic behavior

Does the query count hold for every $f$ in $F_d$, rather than almost every function or nondegenerate instances?

Determinism

Are there no random directions, random rotations, random walks, randomized rounding, Monte Carlo estimates, or probabilistic success events?

Accuracy substitution

After setting $\epsilon=d^{-4}$, do all factors involving $\epsilon$ remain within the claimed dimension bound?

Feasibility and output

Is $x_{\mathrm{hat}}$ in $B_d$, and is the guarantee an objective-value guarantee rather than only distance to a model minimizer, small estimated gradient, or small localization volume?

Uniform constants

Are all constants and sufficiently-large-$d$ conditions explicit enough to establish the claimed asymptotic theorem?

Root-agent responsibilities

The root agent must repeatedly synthesize proved components, challenge all unproved interfaces, redirect resources away from duplicated approaches, and launch new rounds after failures.

Maintain separate ledgers for:

proved upper-bound components;

proved lower-bound components;

exact-value precision issues;

model-alignment issues;

false lemmas and their counterexamples;

blocked theorem-strength gaps;

exponent calculations and polylogarithmic losses.

When a candidate complete proof appears, freeze exploratory expansion temporarily and launch several independent audits. At least one audit should attempt to construct an algorithm violating the proposed lower bound, one should attempt to construct a function violating the proposed upper-bound invariant, one should inspect every precision assumption, and one should recompute all exponents and accuracy scalings from scratch.

If the proof fails audit, record the exact failure and resume diverse search. Do not allow a repaired proof to inherit approval from the failed version; audit the repaired argument again.

Public search and imported results

Public search may be used to retrieve primary sources, verify standard named theorems, check model definitions, and determine whether a claimed recent result actually applies. It may not substitute an abstract, citation, or literature consensus for a proof.

Any imported theorem must be stated with all hypotheses needed here and its role in the argument must be proved. In particular, do not cite a memory-query theorem as a zeroth-order lower bound without proving the exact transcript reduction.

Do not terminate merely because papers describe the problem as open. Conversely, if a published proof is located, reproduce its mathematical argument, verify that it uses the exact oracle and function class above, and subject it to the same adversarial audit.

Final response requirements

Return only after a complete resolution survives audit.

The final response must contain:

a precise theorem identifying $Q_{\mathrm{val}}^{\mathrm{det}}(d,d^{-4})$ up to polylogarithmic factors;

the complete deterministic upper-bound algorithm, including pseudocode or an equivalently precise construction;

a proof of its correctness and complete oracle-call count;

the complete lower-bound theorem for arbitrary deterministic exact-value algorithms;

the hard construction or adversary and proof of fixed-function consistency;

every reduction and precision lemma used;

explicit exponent and accuracy calculations;

a final model-compliance audit addressing every item above.

Do not return a one-sided improvement, reduction, missing lemma, computational result, candidate construction, “best effort” report, or explanation of why the problem is difficult. Do not answer merely that the problem remains open.

Use the full available computation and agent budget. Do not stop after the first wave of approaches fails. Continue launching genuinely new rounds, while reopening blocked routes only when a materially new mechanism is available.

Return only when matching upper and lower bounds for the exact minimax quantity stated above have been proved and the proof has survived independent adversarial audit.

\end{document}